\documentclass{amsart}

\newtheorem{Theo}{Theorem}
\newtheorem{Lem}{Lemma}
\newtheorem{Prop}{Proposition}
\newtheorem{Def}{Definition}

\newcommand{\R}{\mathbb{R}}

\newcommand{\C}{\mathbb{C}}
\newcommand{\N}{\mathbb{N}}

\usepackage{enumerate}  
\usepackage{hyperref}  
\usepackage{color}
\usepackage[utf8]{inputenc}

\begin{document}
\title[Tauberian theorems and the Riemann-Lebesgue Lemma]{Newman's Tauberian theorem, the Riemann-Lebesgue Lemma, and abstract analytic number theory}
\author[J.-C. Schlage-Puchta]{Jan-Christoph Schlage-Puchta}
\author[C. Schwerdt]{Christoph Schwerdt}

\begin{abstract}
We give a generalized and effective version of Bekehermes' improvement of Newman's Tauberian theorem. To do so we prove an effective version of the Riemann-Lebesgue Lemma for functions of bounded $p$-variation. We apply our Tauberian theorem to abstract analytic semigroups and prove a version of the prime number theorem as well as an estimate for Mertens' function with explicit error term. 
\end{abstract}

\maketitle

\tableofcontents

\section{Introduction}

Let $\sum_{k=1}^\infty a_k$ be convergent for a sequence of complex numbers $(a_{k}) \subset \C$.
Furthermore let $G$ be a holomorphic function on a complex set $D \subset \C$ with $B(0,1) \subset D$
that satisfies 
$$
G(z)=\sum_{k=1}^\infty a_k z^k
$$
for every $z \in B(0,1)$. Abel proved that $G$ can be extended continuously to $z=1$ with $G(1) = \sum_{k=1}^\infty a_k$. 
In general the reverse implication is not true, however. Take $G(z) = 1/1+z$ for every 
$z \in \C \setminus {\{-1\}}$ as an example which satisfies
$$
G(z) = \frac{1}{1+z} = \sum_{k=0}^\infty (-1)^{k} z^k
$$
for every $z \in B(0,1)$. Therefore $G$ is characterised as a series of the form $\sum_{k=1}^\infty a_k z^k$ on the unit circle.
But even if $G(1) = 1/2$ is true, there is obviously no convergence of $\sum_{k=1}^\infty a_k$ for $a_{k} = (-1)^{k}$ implied.
However, Tauber proved that the existence of a continuous extension of $G$
to $z=1$ implies the existence of $\sum_{k=1}^\infty a_k$ being equal to $G(1)$, whenever 
$a_k = o(\frac{1}{k})$ is true.\\

In general Tauberian theorems show that a series, which has a value according to some summation method and satisfies additional conditions, has the same value according to some stronger summation method. These conditions usually involve either the summands themselves, as $a_k = o(\frac{1}{k})$ in Tauber's theorem, or the analytic behaviour of a function associated to the series. An important Tauberian theorem is Newman's theorem, which has many applications.\\

\begin{Theo}[Newman]
Let $f \colon [0, \infty) \rightarrow \C$ be bounded and locally integrable. Furthermore define 
$$
g(z)=\int_0^\infty f(t) e^{-zt}\;dt
$$
for $z \in \C$ with $\Re z > 0$. Suppose $g$ can be extended holomorphically 
to some open neighbourhood of $\{z \in \C : \Re z \geq 0 \}$. 
Then $\int_0^\infty f(t)\;dt$ exists and equals $g(0)$.\\
\end{Theo}

Bekehermes \cite{Tobias} improved this result by weakening the conditions of $g$. 
The proof can also be found in the survey article by Riemenschneider \cite{Riemenschneider}.

\begin{Theo}[Bekehermes]
\label{thm:Bekehermes}
Let $f \colon [0, \infty) \rightarrow \C$ be bounded and locally integrable. Define 
$$
g(z)=\int_0^\infty f(t) e^{-zt}\;dt
$$
for $z \in \C$ with $\Re z > 0$. Furthermore suppose $g$ can be extended continuously to the set 
$\{z \in \C : \Re z \geq 0 \}$ and suppose
$$
\underset{\Re z\geq 0}{\lim\limits_{z\rightarrow 0}}\frac{g(z)-g(0)}{z}
$$ 
exists. Then $\int_0^\infty f(t) \;dt$ exists and equals $g(0)$.\\
\end{Theo}

\section{Main results of this article}

The goal of the current work is to weaken the specific conditions on $g$ in Bekehermes' argumentation \cite{Tobias} 
even further and give an effective bound on the speed of convergence of the integral.\\

For locally integrable functions $f \colon [0,\infty) \to \C$ of exponential type, i.e.
$$
\exists C,T > 0, \ z_{0} \in \R \ : \ \left| f(t) \right| \ \leq \ C \mathrm{e}^{z_{0}t}
$$
for every $t \geq T$, the Laplace transform
$$
g(z)= \left( \mathcal{L}f \right) (z) = \int_0^\infty f(t) e^{-zt}\;dt
$$
is well defined on $\left\{ z \in \C \ : \  \Re z > z_{0} \right\}$.
Under the assumption of $f$ being bounded, i.e. $z_{0}=0$ and $T=0$, and additional assumptions on the 
Laplace transform $g$, Newman and Bekehermes implied an existence of
$g(0) = \int_0^\infty f(t) \;dt$.

\begin{Theo}
\label{thm:Laplace}
Let $f \colon [0, \infty) \rightarrow \C$ be bounded and locally integrable. Let the Laplace transform of $f$ be denoted by
$g = \mathcal{L}f$ on $\left\{ z \in \C \ : \  \Re z > 0 \right\}$ and let $g$ satisfy\\
\begin{enumerate}[i.)]
\item $\exists \ \delta_{1}, r, C > 0 : \left| g(z) \right| \leq C$ for almost every $z \in \C$ with $\Re z \in (0,\delta_{1})$ and 
$|\Im z| \geq r$ and
\item $\exists \ \delta_{2}, K > 0 : \left| g(z) \right| \leq K |z|$ for almost every $z \in \C$ with $\Re z > 0$ and $|z| \leq \delta_{2}$.\\
\end{enumerate}

Suppose there is a limit 
$\left\{ t \to \tilde{g}(\mathrm{i}t) \right\} \in \mathrm{L}^{1}\left( \R \right)$ 
such that $g(\varepsilon + \mathrm{i}t)$ converges to 
$\tilde{g}( \mathrm{i}t )$ for $\varepsilon \to 0^{+}$  in $\mathrm{L}_{loc}^{1}\left( \R  \right)$. 
Then the following implications hold: \\
\begin{enumerate}[a)]
\item The function $g$ is continuously extendable from $(0,\infty)$ to $0$ with
$$
g(0) = \int_{0}^{\infty} f(t) \;dt.
$$

\item 
Let $\tilde{g}$ or a representative of  $\tilde{g}$ in 
$\mathrm{L}^{1}\left( \mathbb{R} \right)$ be locally bounded and locally 
of finite $p$-variation both in $(0,\infty)$ and $(-\infty, 0)$ respectively.
For simplicity we treat $\tilde{g}$ as such a function and we simply write $\tilde{g}(t)$ 
instead of $\tilde{g}(\mathrm{i}t)$. Then for any $k \in (0, \frac{1}{2p})$ and sufficiently 
large $R$ and $T$ we obtain
\begin{align*}
& \left| \int_{0}^{T} f(t) \;dt - g(0) \right| & \\
& \ll  \frac{ \| f \|_{\infty}}{R} + \frac{1}{T^{k}} +  
\frac{1}{T^{\frac{1}{p}-2k}} \left( 
\left\| \tilde{g} \right\|_{\mathrm{L}^{\infty}\left( [T^{-k}, R] \right) }^{p} + \left( \frac{V_{p}(\tilde{g})}{T^{k}} \right)^{p}
\right)^{\frac{1}{p}} + \frac{ \left\| \tilde{g} \right\|_{\mathrm{L}^{\infty}( [T^{-k}, R] )} }{T^{1-k}} &  \\
& + \frac{1}{T^{\frac{1}{p}-2k}} \left( 
\left\| \tilde{g} \right\|_{\mathrm{L}^{\infty}\left( [-R, -T^{-k}] \right) }^{p} + \left( \frac{V_{p}(\tilde{g})}{T^{k}} \right)^{p}
\right)^{\frac{1}{p}} + \frac{ \left\| \tilde{g} \right\|_{\mathrm{L}^{\infty}( [-R, -T^{-k}] )} }{T^{1-k}} &  
\end{align*}
where the $p$-variation $V_{p}(\tilde{g})$ of $\tilde{g}$ also refers to the intervalls $[T^{-k}, R]$ and $[-R, -T^{-k}]$ respectively.
\end{enumerate}
\end{Theo}

Remember that the $p$-variation of a function $f \colon [a, b] \to \C$ is defined as the supremum of
$
\sum_{k=1}^n |f(x_{k+1})-f(x_k)|^p
$ 
over all partitions $a\leq x_1<x_2<\dots <x_{n+1}\leq b$. Then $f$ is said to be of finite $p$-variance if this supremum is finite.
Please see the Appendix \ref{appendix} for further details.\\

Our proofs rely on some form of the Riemann-Lebesgue Lemma that for any continuous function 
$f \colon [0, 1] \rightarrow \R$ we have $ \lim_{x \rightarrow \infty} \int_0^1 f(t) e(xt)\;dt = 0$ where we write $e(t)$ 
for $e^{- 2\pi \mathrm{i} t}$. There are various strengthenings of this statement, see e.g. \cite{Kahane}. We need 
effective bounds for the speed of convergence, which we have not found in the literature. Hence we prove the following 
theorem in the Appendix \ref{appendix}.

\begin{Theo}
Let $f \colon [0,1] \rightarrow \R$ be bounded and admit a finite p-variation $V_{p}(f)$. Then 
$$
\left| \int_0^1 f(t) e(xt)\;dt \right| \ \leq \  V_{p}(f)\left( \frac{1}{x} \right)^{\frac{1}{p}}+ \frac{\|f\|_{L^{\infty}}}{x}
$$
holds for any $x > 0$.\\
\end{Theo}

One area of applications of Tauberian theorems is analytic number theory. Newman's theorem can be used to give a 
comparatively simple proof of the prime number theorem. Bekehermes used his theorem to give an analytic proof of 
the prime number theorem for arithmetic semigroups. An arithmetic semigroup is a commutative semigroup $(G,\cdot)$ 
with an identity element together with a multiplicative norm $\| \cdot \| \colon G \rightarrow [1, \infty)$ such that for 
every $x$ the quantity 
$$
N_{G}(x) =\left| \{ g\in G : \| g \|\leq x \} \right|
$$ 
is finite, the identity is the unique element satisfying $\|g\|=1$ and the decomposition into irreducible elements is unique.
The natural numbers $\N$ with the usual multiplication of two elements and the absolute value $| \cdot |$ is a simple 
example of an arithmetic semigroup. By $\pi_G(x)$ we denote the number of irreducible elements in $G$ of norm smaller 
or equal to $x$. Now, we use Theorem \ref{thm:Laplace} to prove the following.\\

\begin{Theo}
\label{thm:PNT}
Let $G$ be an aithmetic semigroup satisfying $N_{G}(x)=Ax+\mathcal{O}\left(\frac{x}{\log^\gamma x}\right)$ with $\gamma>2$
and $A \in \R$. Then we have
\[
\pi_G(x) = \frac{x}{\log x} + \mathcal{O}\left(\frac{x}{\log^{1+\delta} x}\right),
\]
where $\delta = \min\left(\frac{\gamma-2}{6}, \frac{1}{198}\right)$.\\
\end{Theo}

We did not bother with the pretty bad constant $1/198$, the important observation is that the exponent in the error term 
grows linearly with $\gamma-2$. In a similar way we can prove the following. Let $G$ be an arithmetic semigroup.  
Define the Möbiusfunction $\mu \colon G \rightarrow \{0, \pm 1\}$ by 
$$
\mu(g)= \left\{\begin{array}{cl} 
(-1)^{k} & \text{if $g$ is the product of $k$ different irreducible elements} \\
0 & \text{otherwise}
\end{array}\right.
$$
and $M_G(x)=\sum_{\|g\|\leq x}\mu_G(g)$. In the classical setting there is a simple and elementary proof that the 
prime number theorem is equivalent to the statement of $M_{G}(x)=o(x)$. In the abstract setting this equivalence 
does not hold, however.\\

\begin{Theo}
\label{thm:Moebius}
Let $G$ be an aithmetic semigroup satisfying $N_{G}(x)=Ax+\mathcal{O}\left(\frac{x}{\log^\gamma x}\right)$ 
with $\gamma>\frac{3}{2}$ and $A \in \R$. Then we have $M_{G}(x)\ll\frac{x}{\log^\delta x}$ with 
$\delta=\frac{(\gamma-\frac{3}{2})^2}{72}$.\\
\end{Theo}

Both Theorem~\ref{thm:PNT} and \ref{thm:Moebius} are obtained by applying Theorem~\ref{thm:Laplace} to the 
zeta function $\zeta$ of the semigroup $G$ defined as $\zeta(s)=\sum_{g\in G} \frac{1}{\|g\|^s}$
for $s \in \mathbb{C}$ with $\Re s > 1$. 
Comparing Theorem~\ref{thm:PNT} and \ref{thm:Moebius} we see that while in the former the exponent in the 
error term grows linearly with $\gamma-2$, in the latter the exponent grows only quadratically in $\gamma-\frac{3}{2}$. 
The reason is that for the former we have to consider the boundary behaviour of $\frac{\zeta'}{\zeta}$ on the line 
$\{z:\Re z=1\}$, while for the latter we have to consider $\frac{1}{\zeta(s)}$. Now $\gamma>\frac{3}{2}$ implies that 
$\zeta(1+it)\neq 0$, while $\gamma>2$ implies that $\zeta'(s)$ has a continuous extension to the line $\{z:\Re z=1\}$. 
This leads to the fact that in Theorem~\ref{thm:PNT} the main contribution to the error term comes from large 
oscillations of $\zeta'(1+it)$, whereas for Theorem~\ref{thm:Moebius} the error term is dominated by possible 
values where $\zeta(1+it)$ is very small. The first problems can be solved reasonably well using partial summation, 
while the second requires more intricate arguments, therefore, the error term in the second case is 
worse.

\section{Proofs}

\subsection{Proof of Theorem \ref{thm:Laplace}}

In the following we present an effective convergence result for Laplace transforms at $z = 0$. For a better understanding we
devided the argumentation into several steps.\\
\begin{enumerate}[a)]
\item
\begin{enumerate}[i.)]
\item For $z \in \C$ with $\Re z > 0$ the Laplace transform $g$ of $f$ is given by
$$
g(z)=\int_0^\infty f(t) e^{-zt}\;dt
$$
since $f$ is bounded. Let $T>0$ be arbitrary but fixed and define 
$$
g_{T}(z)=\int_{0}^{T} f(t) e^{-tz}\;dt
$$
for any $z \in \C$. In particular we have $g_{T}(0) = \int_{0}^{T} f(t) \;dt$ for $z=0$. In the following we will prove that
$g$ is continuously extendable from $(0,\infty)$ to $0$ and that $g(0)$ is equal to the limit of $g_{T}(0)$ for $T \to \infty$.\\
\item Let $\sigma \in (0,1/2)$ and $R > \max (r, 1)$ be arbitrary but fixed. Choose any $\varepsilon \in (0, \sigma/2)$ and
define\\
\begin{enumerate}[a)]
\item $B_{R}^{+}(\varepsilon) = \partial B_{R}(\varepsilon) \cap \left\{  z \in \C \ : \ \Re z > \varepsilon \right\}$ and
\item $D_{R}(\varepsilon) = \left\{ \varepsilon + \mathrm{i}t \ : \ t \in [-R, R] \right\}$. \\
\end{enumerate}
Using Cauchy's theorem we conclude to
\begin{align*}
g(\sigma) & = \frac{1}{2\pi \mathrm{i}} 
\int_{D_{R}(\varepsilon) \cup B_{R}^{+}(\varepsilon)} \frac{g(z)}{z-\sigma} \;dz & \\
& = \frac{1}{2\pi \mathrm{i}}  \int_{ D_{R}(\varepsilon) \cup B_{R}^{+}(\varepsilon) }
g(z) \mathrm{e}^{T(z-\sigma)}\left(\frac{1}{z-\sigma}+\frac{z}{R^2}\right)\;dz &
\end{align*}
since $g$ is holomorphic in $\left\{ z \in \C \ : \ \Re z > 0 \right\}$. Note that for $z=\sigma$ 
we have  $e^{T(z-\sigma)}=1$, so the factor $e^{T(z-\sigma)}$ does not change the residuum.
Furthermore $zR^{-2}$ is an entire function, so adding this summand does not change the integral either. \\

\item Next we want to push the path of integration to $\varepsilon = 0$. Keep in mind that\\
\begin{align*}
& \frac{1}{2\pi \mathrm{i}} \int_{ D_{R}(\varepsilon) } g(z) \mathrm{e}^{T(z-\sigma)}\left( \frac{1}{z-\sigma}+\frac{z}{R^2} \right)\;dz
 & \\
& = \frac{1}{2\pi} \int_{-R}^{R} 
g(\varepsilon + \mathrm{i}t) \mathrm{e}^{T(\varepsilon - \sigma + \mathrm{i}t)} 
\left(  \frac{1}{\varepsilon - \sigma + \mathrm{i}t}+\frac{\varepsilon + \mathrm{i}t}{R^2} \right) \;dt & \\
\end{align*}
holds for the parametrization $\gamma_{\varepsilon}(t) = \varepsilon + \mathrm{i}t$ with 
$t \in [- R, R]$. For $\varepsilon \to 0^{+}$ the integral on the right hand side of the last equation converges to 
$$
 \frac{1}{2\pi} \int_{- R}^{R} \tilde{g}( \mathrm{i}t) \mathrm{e}^{T( - \sigma + \mathrm{i}t)} 
\left(  \frac{1}{ - \sigma + \mathrm{i}t}+\frac{\mathrm{i}t}{R^2} \right) \;dt
$$
since $g(\varepsilon + \mathrm{i}t)$ converges to 
$\tilde{g}( \mathrm{i}t)$ in $\mathrm{L}^{1}\left( [-R, R] \right)$ while the absolute values of the
remaining terms are uniformly bounded for $\varepsilon \in [0,\sigma/2)$ and $t \in [-R,R]$. Next, let us focus 
on the integral along $B_{R}^{+}(\varepsilon)$ with\\
 \begin{align*}
& \frac{1}{2\pi \mathrm{i}} \int_{B_{R}^{+}(\varepsilon)} g(z) \mathrm{e}^{T(z-\sigma)}\left( \frac{1}{z-\sigma}+\frac{z}{R^2} \right)\;dz
 & \\
& = \frac{R}{2\pi} \int_{\frac{3}{2}\pi}^{\frac{5}{2}\pi} 
g(\varepsilon + R\mathrm{e}^{\mathrm{i}t}) \mathrm{e}^{T(\varepsilon - \sigma + R\mathrm{e}^{\mathrm{i}t})} 
\left(  \frac{1}{\varepsilon - \sigma + R\mathrm{e}^{\mathrm{i}t}}+\frac{\varepsilon + R\mathrm{e}^{\mathrm{i}t}}{R^2} \right) 
\mathrm{e}^{\mathrm{i}t} \;dt & \\
\end{align*}
for the parametrization $\tilde{\gamma}_{\varepsilon}(t) = \varepsilon + R\mathrm{e}^{\mathrm{i}t}$ with 
$t \in [\frac{3}{2}\pi, \frac{5}{2}\pi]$. For $z \in B_{R}^{+}(\varepsilon)$ we argue that
$$
|z - \sigma| \geq \left| (\varepsilon + R) - \sigma \right| =  R - ( \sigma - \varepsilon ) > R - \sigma > R - \frac{1}{2} 
> 1 - \frac{1}{2} = \frac{1}{2}
$$
due to geometric arguments for $\sigma > \varepsilon$ being contained in $B_{R}(\varepsilon)$. For $z \in B_{R}^{+}(\varepsilon)$ with $\Re z > \delta_{1}$
we infer that
$$
\left| \frac{g(z)}{z - \sigma}\right| \leq 2 | g(z) | \leq \frac{ 2 \| f \|_{\infty}}{\Re z} \leq  \frac{ 2 \| f \|_{\infty}}{\delta_{1}}
$$
is true. Furthermore for $\varepsilon \in (0,\delta_{1})$ and $R$ sufficiently large we conclude that $z \in B_{R}^{+}(\varepsilon)$
with $\Re z \leq \delta_{1}$ always implies $\left| \Im z \right| \geq r$. In that case we conclude to 
$\left| g(z)(z - \sigma)^{-1} \right| \leq 2 C$. Hence for every $z \in B_{R}^{+}(\varepsilon)$ we infer that
$$
\left| \frac{g(z)}{z - \sigma}\right| \leq  \frac{ 2 \| f \|_{\infty}}{\delta_{1}} + 2C
$$
is true. Therefore we conclude to the convergence of
 \begin{align*}
&  \int_{\frac{3}{2}\pi}^{\frac{5}{2}\pi} 
g(\varepsilon + R\mathrm{e}^{\mathrm{i}t}) \mathrm{e}^{T(\varepsilon - \sigma + R\mathrm{e}^{\mathrm{i}t})} 
\left(  \frac{1}{\varepsilon - \sigma + R\mathrm{e}^{\mathrm{i}t}}+\frac{\varepsilon + R\mathrm{e}^{\mathrm{i}t}}{R^2} \right) 
\mathrm{e}^{\mathrm{i}t} \;dt & \\
& \to  
\int_{\frac{3}{2}\pi}^{\frac{5}{2}\pi} 
g( R\mathrm{e}^{\mathrm{i}t}) \mathrm{e}^{T( - \sigma + R\mathrm{e}^{\mathrm{i}t})} 
\left(  \frac{1}{ - \sigma + R\mathrm{e}^{\mathrm{i}t}}+\frac{\mathrm{e}^{\mathrm{i}t}}{R} \right) 
\mathrm{e}^{\mathrm{i}t} \;dt
\end{align*}
for $\varepsilon \to 0^{+}$ by the dominated convergence theorem. Note that the convergence of 
$g(\varepsilon + R\mathrm{e}^{\mathrm{i}t})$ to $g( R\mathrm{e}^{\mathrm{i}t})$ for $\varepsilon \to 0^{+}$
and $t \in (\frac{3}{2} \pi, \frac{5}{2} \pi)$ is not trivial to see. Again we argue with the dominated convergence theorem that
\begin{align*}
\left| g(\varepsilon + R\mathrm{e}^{\mathrm{i}t}) - g( R\mathrm{e}^{\mathrm{i}t})  \right|
& =  \left| \int_{0}^{\infty} f(s) \left( \mathrm{e}^{-\left( \varepsilon + R\mathrm{e}^{\mathrm{i}t} \right)s }
- \mathrm{e}^{-R\mathrm{e}^{\mathrm{i}t}s } \right) \;ds \right| & \\
& \leq \left\| f \right\|_{\infty} \ \int_{0}^{\infty} \mathrm{e}^{- sR \cos(t)} \left| \mathrm{e}^{-\varepsilon s} -1 \right|  \;ds
\to 0
\end{align*}
is implied for $\varepsilon \to 0^{+}$ since $\mathrm{e}^{- sR \cos(t)} \in \mathrm{L}^{1}\left( [0,\infty) \right)$ as $\cos(t) > 0$ 
is true for $t \in (\frac{3}{2} \pi, \frac{5}{2} \pi)$ and $ \left| \mathrm{e}^{-\varepsilon s} -1 \right| \leq 2$.
So, in total we end up with \\
$$
g(\sigma) = \frac{1}{2\pi \mathrm{i}} \left( 
 \int_{D_{R}}
\tilde{g}(z) \mathrm{e}^{T(z-\sigma)}\left(\frac{1}{z-\sigma}+\frac{z}{R^2}\right)\;dz \ + \
 \int_{B_{R}^{+}(0)} g(z) \mathrm{e}^{T(z-\sigma)}\left(\frac{1}{z-\sigma}+\frac{z}{R^2}\right)\;dz  \right)
$$
for $D_{R} = \{ \mathrm{i}t \ : \ t \in [-R, R]  \}$ and 
 $B_{R}^{+}(0) = \partial B_{R}(0) \cap \left\{  z \in \C \ : \ \Re z > 0 \right\}$.\\

\item This last equation holds for every $\sigma \in (0, 1/2)$ and sufficiently large $R$. Next we prove that $g$ is continuously extendable 
from $(0,\infty)$ to $0$. Again 
$$
\int_{B_{R}^{+}(0)} g(z) \mathrm{e}^{T(z-\sigma)}\left(\frac{1}{z-\sigma}+\frac{z}{R^2}\right)\;dz
\to \int_{B_{R}^{+}(0)} g(z) \mathrm{e}^{Tz}\left(\frac{1}{z}+\frac{z}{R^2}\right)\;dz
$$
for $\sigma \to 0^{+}$ is implied by the dominated convergence theorem as all terms are uniformly bounded
on $B_{R}^{+}(0)$. The argumentation refering to second integral along $D_{R}$ is more 
complicated. First we need to prove that
$$
\left| \tilde{g}(\mathrm{i}t) \right| \leq K |t|
$$ 
holds for almost every $t \in [-\delta_{2}, \delta_{2}]$. Therefore let 
$(\varepsilon_{n}) \subset (0,\infty)$ be a sequence with $\varepsilon_{n} \to 0$ for $n \to \infty$. Then
$g(\varepsilon_{n} + \mathrm{i}t)$ converges to $\tilde{g}(\mathrm{i}t)$ in $\mathrm{L}^{1}\left( [-\delta_{2}, \delta_{2}]\right)$
for $n \to \infty$. Hence there exists a subsequence $(\varepsilon_{n_{k}}) \subset (\varepsilon_{n})$ such that
$g(\varepsilon_{n_{k}} + \mathrm{i}t)$ converges to $\tilde{g}(\mathrm{i}t)$ pointwise almost everywhere in 
$[-\delta_{2}, \delta_{2}]$. Hence for almost every $t \in [-\delta_{2}, \delta_{2}]$ we conclude to
$$
\left| \tilde{g}( \mathrm{i}t ) \right| = \lim_{k \to \infty} \left| g(\varepsilon_{n_{k}} + \mathrm{i}t) \right| \leq K |t|.
$$
So for $t$ close to $0$ we conclude to 
$$
\frac{\left| \tilde{g}(\mathrm{i}t) \right|}{\left| \mathrm{i}t - \sigma \right|} \ \leq \ K \frac{|t|}{|t|} = K.
$$
Using the dominated convergence theorem we infer that
$$
\int_{D_{R}} \tilde{g}(z) \mathrm{e}^{T(z-\sigma)}\left(\frac{1}{z-\sigma}+\frac{z}{R^2}\right)\;dz
\to \int_{D_{R}} \tilde{g}(z) \mathrm{e}^{Tz}\left(\frac{1}{z}+\frac{z}{R^2}\right)\;dz
$$
is implied by $\sigma \to 0^{+}$. Therefore $g$ is continuously extendable from $(0,\infty)$ to $0$ satisfying
$$
g(0) = \frac{1}{2\pi \mathrm{i}} \left(  \int_{D_{R}}
\tilde{g}(z) \mathrm{e}^{Tz}\left(\frac{1}{z}+\frac{z}{R^2}\right)\;dz \ + \  
\int_{B_{R}^{+}(0)} g(z) \mathrm{e}^{Tz}\left(\frac{1}{z}+\frac{z}{R^2}\right)\;dz \right).
$$

\item Furthermore 
$$
g_{T}(0) = \frac{1}{2\pi \mathrm{i}} \int_{B_{R}(0)}
g_{T}(z) e^{Tz} \left( \frac{1}{z}+\frac{z}{R^2} \right) \;dz
$$
is implied by Cauchy's theorem since $g_{T}$ is an entire function. Both
characterizations of $g$ and $g_{T}$ share the common path of integration 
$B_{R}^{+}(0)$ and hence we have
\begin{eqnarray*}
\left| \int_{0}^{T} f(t) \;dt - g(0)  \right| & = & \left| g_{T}(0) - g(0) \right| \\
 & \leq & \left|  \frac{1}{2\pi \mathrm{i}} \int_{ B_{R}^{+}(0) } ( g_{T}(z) - g(z) ) e^{Tz}\left(\frac{1}{z}+\frac{z}{R^2}\right)\;dz \right| \\
 & & + \left|  \frac{1}{2\pi \mathrm{i}}  \int_{ D_{R} } \tilde{g}(z) e^{Tz} \left( \frac{1}{z}+\frac{z}{R^2} \right) \;dz \right| \\
& & + \left|  \frac{1}{2\pi \mathrm{i}}  \int_{ B_{R}^{-}(0) } g_T(z) e^{Tz}\left(\frac{1}{z}+\frac{z}{R^2}\right)\;dz \right| 
\end{eqnarray*}
with $B_{R}^{-}(0) = \partial B_{R}(0) \cap \left\{ z \in \C \ : \ \Re z < 0 \right\}$. Now we estimate each of 
these integrals separately. If $|z|=R$, then 
$$
\frac{1}{z}+\frac{z}{R^2} = \frac{1}{z}+\frac{z}{|z|^2} = \frac{\overline{z} + z}{|z|^2} = \frac{2 \Re z}{|z|^2}
$$
is true. Furthermore for $z \in B_{R}^{+}(0)$ with $\Re z > 0$ we infer that
$$
| g_{T}(z) - g(z) | = 
\left| \int_{T}^{\infty} f(t)e^{-zt} \;dt \right| \leq 
\| f \|_{\infty} \int_{T}^{\infty} e^{-t \Re z} \;dt = \frac{ \| f \|_{\infty} }{ \Re z } e^{- T \Re z}
$$
is true. Using this bound we obtain \\
\begin{align*}
& \left| \int_{ B_{R}^{+}(0) } ( g_{T}(z) - g(z) ) e^{Tz}\left(\frac{1}{z}+\frac{z}{R^2}\right)\;dz \right| & \\
& & \\
& \leq  \int_{ B_{R}^{+}(0) }\frac{ \| f \|_{\infty} }{ \Re z } e^{- T \Re z} e^{T \Re z} \frac{2\Re z}{|z|^2} \;dz
 = \| f \|_{\infty} \int_{ B_{R}^{+}(0) } \frac{dz}{|z|^2} \ll \frac{ \| f \|_{\infty} }{R}. \\
\end{align*}
For $z \in B_{R}^{-}(0)$ with $\Re z < 0$ we infer that
$$
\left| g_{T}(z) \right| \leq 
\| f \|_{\infty} \int_{0}^{T} \mathrm{e}^{-t \Re z}  \;dt = \frac{ \| f \|_{\infty} }{ - \Re z} \left( \mathrm{e}^{T(- \Re z)} - 1 \right)
\leq \frac{ \| f \|_{\infty} }{ - \Re z} \ \mathrm{e}^{T(- \Re z)} 
$$
is true. Therefore we obtain \\
\begin{align*}
& \left| \int_{ B_{R}^{-}(0) } g_{T}(z) e^{Tz}\left(\frac{1}{z}+\frac{z}{R^2}\right)\;dz \right| & \\
& & \\
& \leq  \int_{ B_{R}^{-}(0) }\frac{ \| f \|_{\infty} }{ -\Re z } e^{ - T \Re z} e^{T \Re z} \frac{2 ( - \Re z)}{|z|^2} \;dz
 = \| f \|_{\infty} \int_{ B_{R}^{-}(0) } \frac{dz}{|z|^2} \ll \frac{ \| f \|_{\infty} }{R}. \\
\end{align*}
Let us consider the integral along $D_{R}$ next. For $\varepsilon \in (0, \delta_{2})$ arbitrary but fixed we write
\begin{align*}
& \left| \int_{-R}^{R} \tilde{g}(\mathrm{i}t) \mathrm{e}^{\mathrm{i}tT} 
\left( \frac{1}{\mathrm{i}t} + \frac{\mathrm{i}t}{R^{2}} \right) \;dt \right| \leq & \\
& & \\
& \left| \int_{\varepsilon}^{R} \tilde{g}(\mathrm{i}t) \left( \frac{1}{\mathrm{i}t} + \frac{\mathrm{i}t}{R^{2}} \right) 
\mathrm{e}^{\mathrm{i}tT}  \;dt \right|  + 
\int_{-\varepsilon}^{\varepsilon} \left|  \tilde{g}(\mathrm{i}t) \left( \frac{1}{\mathrm{i}t} + \frac{\mathrm{i}t}{R^{2}} \right) \right|  \;dt  
+ \left| \int_{- R}^{- \varepsilon} \tilde{g}(\mathrm{i}t) \left( \frac{1}{\mathrm{i}t} + \frac{\mathrm{i}t}{R^{2}} \right) 
\mathrm{e}^{\mathrm{i}tT}  \;dt \right|.  & \\
\end{align*}
For $\varepsilon$ being smaller than $\delta_{2}$ we infer that 
$\left\{ t \to  \tilde{g}(\mathrm{i}t) \left( \frac{1}{\mathrm{i}t} + \frac{\mathrm{i}t}{R^{2}} \right) \right\}$
is bounded in $[-\delta_{2}, \delta_{2}]$. Hence 
$$
\int_{-\varepsilon}^{\varepsilon}  
\left|  \tilde{g}(\mathrm{i}t) \left( \frac{1}{\mathrm{i}t} + \frac{\mathrm{i}t}{R^{2}} \right) \right|  \;dt \ll \varepsilon
$$
is implied. Next we focus on 
$$
\int_{\varepsilon}^{R} \tilde{g}(\mathrm{i}t) \left( \frac{1}{\mathrm{i}t} + \frac{\mathrm{i}t}{R^{2}} \right) 
\mathrm{e}^{\mathrm{i}tT}  \;dt.
$$
Note that $\left\{ t \to  \tilde{g}(\mathrm{i}t) \left( \frac{1}{\mathrm{i}t} + \frac{\mathrm{i}t}{R^{2}} \right) \right\}$
is contained in $\mathrm{L}^{1}\left( [\varepsilon, R] \right)$ since the term 
$\left( \frac{1}{\mathrm{i}t} + \frac{\mathrm{i}t}{R^{2}} \right)$ 
is bounded in $[\varepsilon, R]$. Then there exists $g_{R}$ a step function 
on $[\varepsilon, R]$ such that
$$
\int_{\varepsilon}^{R} \left|  \tilde{g}(\mathrm{i}t) \left( \frac{1}{\mathrm{i}t} + \frac{\mathrm{i}t}{R^{2}} \right) - g_{R}(t) \right| \;dt 
\leq \frac{1}{R}
$$
is true since the set of step functions is a dense sub-space of $\mathrm{L}^{1}\left( [\varepsilon, R] \right)$.
We use $g_{R}$ to infer that
\begin{align*}
& \left| \int_{\varepsilon}^{R} \tilde{g}(\mathrm{i}t) \left( \frac{1}{\mathrm{i}t} + \frac{\mathrm{i}t}{R^{2}} \right) 
\mathrm{e}^{\mathrm{i}tT}  \;dt \right| & \\
& \leq \int_{\varepsilon}^{R} \left|  \tilde{g}(\mathrm{i}t) \left( \frac{1}{\mathrm{i}t} + \frac{\mathrm{i}t}{R^{2}} \right) 
- g_{R}(t)  \right| \;dt + \left| \int_{\varepsilon}^{R} g_{R}(t) \mathrm{e}^{\mathrm{i}tT}  \;dt \right| & \\
& \leq \frac{1}{R} +  \left| \int_{\varepsilon}^{R} g_{R}(t) \mathrm{e}^{\mathrm{i}tT}  \;dt \right| & 
\end{align*}
is true. Remember that step functions always admit a finite $p$-variation for any $p \geq 1$.\\

\item We use Theorem \ref{Hoelder_variation} on $g_{R}$ next. Therefore we have to make some minor 
adjustments first. Note that
$$
\int_{\varepsilon}^{R} g_{R}(t) \mathrm{e}^{\mathrm{i}tT}  \;dt = 
(R - \varepsilon)  \mathrm{e}^{\mathrm{i} \varepsilon T}
\int_{0}^{1} g_{R}( (R - \varepsilon) r + \varepsilon) \mathrm{e}^{\mathrm{i} (R - \varepsilon) T r}  \;dr
$$ 
is true. The function $\left\{ r \to g_{R}( (R - \varepsilon) r + \varepsilon) \right\}$ is a step function on $[0,1]$ and 
is of finite $p$-variation for any $p \geq 1$. For simplicity we choose $p=1$ and conclude to 
\begin{align*}
\left| \int_{\varepsilon}^{R} g_{R}(t) \mathrm{e}^{\mathrm{i}tT}  \;dt \right| & = (R - \varepsilon) 
\left| \int_{0}^{1} g_{R}( (R - \varepsilon) r + \varepsilon) \mathrm{e}^{\mathrm{i} (R - \varepsilon) T r}  \;dr \right| & \\
& \leq (R - \varepsilon) \left(  \frac{ V_{p}( g_{R} ) }{(R - \varepsilon) T} + 
\frac{ \| g_{R} \|_{\mathrm{L}^{\infty}([\varepsilon, R])} }{ (R - \varepsilon) T }  \right) & \\ 
& = \frac{1}{T} \left( V_{p}(g_{R}) + \| g_{R} \|_{\mathrm{L}^{\infty}([\varepsilon, R])} \right). & 
\end{align*}
In total we conclude to
$$
\left| \int_{\varepsilon}^{R} \tilde{g}(\mathrm{i}t) \left( \frac{1}{\mathrm{i}t} + \frac{\mathrm{i}t}{R^{2}} \right) 
\mathrm{e}^{\mathrm{i}tT}  \;dt \right| \leq \frac{1}{R} + 
\frac{1}{T} \left( V_{p}(g_{R}) + \| g_{R} \|_{\mathrm{L}^{\infty}([\varepsilon, R])} \right).
$$
Similar arguments are used for 
$\int_{- R}^{- \varepsilon} \tilde{g}(\mathrm{i}t) \left( \frac{1}{\mathrm{i}t} + \frac{\mathrm{i}t}{R^{2}} \right) 
\mathrm{e}^{\mathrm{i}tT}  \;dt$.\\

\item For $\varepsilon \in (0, \delta_{2})$ we summarize our arguments to \\
\begin{align*}
& \left| \int_{0}^{T} f(t) \;dt - g(0) \right| & \\
& & \\
& \leq  \frac{C_{1} \| f \|_{\infty}}{R} + 
\left|  \frac{1}{2\pi \mathrm{i}}  \int_{ D_{R} } \tilde{g}(z) e^{Tz} \left( \frac{1}{z}+\frac{z}{R^2} \right) \;dz \right| & \\ 
& & \\
& \leq   \frac{C_{1} \| f \|_{\infty}}{R} + 
\left| \int_{\varepsilon}^{R} \tilde{g}(\mathrm{i}t) \left( \frac{1}{\mathrm{i}t} + \frac{\mathrm{i}t}{R^{2}} \right) 
\mathrm{e}^{\mathrm{i}tT}  \;dt \right|  + 
\int_{-\varepsilon}^{\varepsilon} \left|  \tilde{g}(\mathrm{i}t) \left( \frac{1}{\mathrm{i}t} + \frac{\mathrm{i}t}{R^{2}} \right) \right|  \;dt  
+ \left| \int_{- R}^{- \varepsilon} \tilde{g}(\mathrm{i}t) \left( \frac{1}{\mathrm{i}t} + \frac{\mathrm{i}t}{R^{2}} \right) 
\mathrm{e}^{\mathrm{i}tT}  \;dt \right| & \\
& & \\
& \leq   \frac{C_{1} \| f \|_{\infty}}{R} + C_{2} \varepsilon +  
\int_{\varepsilon}^{R} \left|  \tilde{g}(\mathrm{i}t) \left( \frac{1}{\mathrm{i}t} + \frac{\mathrm{i}t}{R^{2}} \right) 
- g_{R}(t)  \right| \;dt + \left| \int_{\varepsilon}^{R} g_{R}(t) \mathrm{e}^{\mathrm{i}tT}  \;dt \right| 
+ \left| \int_{- R}^{- \varepsilon} \tilde{g}(\mathrm{i}t) \left( \frac{1}{\mathrm{i}t} + \frac{\mathrm{i}t}{R^{2}} \right) 
\mathrm{e}^{\mathrm{i}tT}  \;dt \right| & \\
& & \\
& \leq \frac{C_{1} \| f \|_{\infty}}{R} +  C_{2} \varepsilon +  \frac{1}{R} + 
\frac{1}{T} \left( V_{p}(g_{R}) + \| g_{R} \|_{\mathrm{L}^{\infty}([\varepsilon, R])} \right) 
+ \left| \int_{- R}^{- \varepsilon} \tilde{g}(\mathrm{i}t) \left( \frac{1}{\mathrm{i}t} + \frac{\mathrm{i}t}{R^{2}} \right) 
\mathrm{e}^{\mathrm{i}tT}  \;dt \right| & \\
\end{align*}
for constants $C_{1}, C_{2} > 0$. Finally we have 
$$
\left| \int_{0}^{T} f(t) \;dt - g(0) \right| \leq \frac{C_{1} \| f \|_{\infty}}{R} + C_{2} \varepsilon +  \frac{2}{R} + 
\frac{2}{T} \left( V_{p}(g_{R}) + \| g_{R} \|_{\infty} \right)
$$
where we used the argumentation in vi.) on 
$\int_{- R}^{- \varepsilon} \tilde{g}(\mathrm{i}t) \left( \frac{1}{\mathrm{i}t} + \frac{\mathrm{i}t}{R^{2}} \right) 
\mathrm{e}^{\mathrm{i}tT}  \;dt$ as well. We choose $R= R(\varepsilon)$ sufficiently large such that 
$$
\frac{C_{1} \| f \|_{\infty}}{R} + \frac{2}{R} < \varepsilon
$$
is true. Then we choose $T > 0$ depending on $R$ and $\varepsilon$ such that 
$$
\frac{2}{T} \left( V_{p}(g_{R}) + \| g_{R} \|_{\infty} \right) < \varepsilon
$$ 
which finally implies $g(0) = \lim_{T \to \infty} \int_{0}^{T} f(t) \;dt$.\\
\end{enumerate}

\item 
Let $\tilde{g}$ itself or a representative of  $\tilde{g}$ in 
$\mathrm{L}^{1}\left( \mathbb{R} \right)$ be locally bounded and locally 
of finite $p$-variation both in $(0,\infty)$ and $(-\infty, 0)$ respectively. For simplicity we treat 
$\tilde{g}$ as such a function. For sufficiently small but fixed $\varepsilon > 0$ we argue similar to a) that 
\begin{align*}
& \left| \int_{0}^{T} f(t) \;dt - g(0) \right| & \\
& & \\
& \leq  \frac{C_{1} \| f \|_{\infty}}{R} + 
\left|  \frac{1}{2\pi \mathrm{i}}  \int_{ D_{R} } \tilde{g}(z) e^{Tz} \left( \frac{1}{z}+\frac{z}{R^2} \right) \;dz \right| & \\ 
& & \\
& \leq   \frac{C_{1} \| f \|_{\infty}}{R} + 
\left| \int_{\varepsilon}^{R} \tilde{g}(\mathrm{i}t) \left( \frac{1}{\mathrm{i}t} + \frac{\mathrm{i}t}{R^{2}} \right) 
\mathrm{e}^{\mathrm{i}tT}  \;dt \right|  + 
\int_{-\varepsilon}^{\varepsilon} \left|  \tilde{g}(\mathrm{i}t) \left( \frac{1}{\mathrm{i}t} + \frac{\mathrm{i}t}{R^{2}} \right) \right|  \;dt  
+ \left| \int_{- R}^{- \varepsilon} \tilde{g}(\mathrm{i}t) \left( \frac{1}{\mathrm{i}t} + \frac{\mathrm{i}t}{R^{2}} \right) 
\mathrm{e}^{\mathrm{i}tT}  \;dt \right| & \\
& & \\
& \leq   \frac{C_{1} \| f \|_{\infty}}{R} + C_{2} \varepsilon +
\left| \int_{\varepsilon}^{R} \tilde{g}(\mathrm{i}t) \left( \frac{1}{\mathrm{i}t} + \frac{\mathrm{i}t}{R^{2}} \right) 
\mathrm{e}^{\mathrm{i}tT}  \;dt \right| 
+ \left| \int_{- R}^{- \varepsilon} \tilde{g}(\mathrm{i}t) \left( \frac{1}{\mathrm{i}t} + \frac{\mathrm{i}t}{R^{2}} \right) 
\mathrm{e}^{\mathrm{i}tT}  \;dt \right| & 
\end{align*}
is true. We define $h(t) = \frac{1}{\mathrm{i}t} + \frac{\mathrm{i}t}{R^{2}}$ for $t \in \left[ \varepsilon, R \right]$ and use
Theorem \ref{Hoelder_variation} in combination with Lemma \ref{Lem:products} for the term
$$
\left| \int_{\varepsilon}^{R} \tilde{g}(\mathrm{i}t) \left( \frac{1}{\mathrm{i}t} + \frac{\mathrm{i}t}{R^{2}} \right) 
\mathrm{e}^{\mathrm{i}tT}  \;dt \right| = 
\left| \int_{\varepsilon}^{R} \tilde{g}(\mathrm{i}t) h(t) \mathrm{e}^{\mathrm{i}tT}  \;dt \right|.
$$
We infer that 
$$
\| h \|_{\mathrm{L}^{\infty} \left( [\varepsilon, R] \right)} \leq \left( \frac{1}{R} + \frac{1}{\varepsilon} \right) 
\leq \frac{2}{\varepsilon}
$$
holds for $R$ being sufficiently large. Similarily we argue that 
$$
\| h^{\prime} \|_{\mathrm{L}^{\infty}\left( [\varepsilon, R]\right)} \leq \frac{2}{\varepsilon^{2}}
$$ 
is true for large $R$. Now we use Theorem \ref{Hoelder_variation} to conclude to
\begin{align*}
\left| \int_{\varepsilon}^{R} \tilde{g}(\mathrm{i}t) \left( \frac{1}{\mathrm{i}t} + \frac{\mathrm{i}t}{R^{2}} \right) 
\mathrm{e}^{\mathrm{i}tT}  \;dt \right| & = 
\left| \int_{\varepsilon}^{R} \tilde{g}(\mathrm{i}t) h(t) \mathrm{e}^{\mathrm{i}tT}  \;dt \right| & \\
& & \\
& \leq \frac{1}{T^{\frac{1}{p}}} V_{p}(\tilde{g}h) + \frac{1}{T} \left\| \tilde{g} \right\|_{\mathrm{L}^{\infty}([\varepsilon, R])} 
\left\| h \right\|_{\mathrm{L}^{\infty}([\varepsilon, R])} & \\
& & \\
& \leq \frac{1}{T^{\frac{1}{p}}} V_{p}(\tilde{g}h) + \frac{2}{T \varepsilon}  \left\| \tilde{g} \right\|_{\mathrm{L}^{\infty}([\varepsilon, R])}. 
\end{align*}
Using Lemma \ref{Lem:products} we argue further that
\begin{align*}
V_{p}(\tilde{g}h) & \leq 2 \left( \left\| \tilde{g} \right\|_{\mathrm{L}^{\infty}\left( [\varepsilon, R]\right) }^{p} 
\left\| h^{\prime} \right\|_{\mathrm{L}^{\infty}\left( [\varepsilon, R]\right) }^{p} +
\left\| h \right\|_{\mathrm{L}^{\infty}\left( [\varepsilon, R]\right) }^{p} V_{p}(\tilde{g})^{p}
\right)^{\frac{1}{p}} & \\
& \leq 2 \left( \left\| \tilde{g} \right\|_{\mathrm{L}^{\infty}\left( [\varepsilon, R]\right) }^{p} 
\frac{2^{p}}{\varepsilon^{2p}} + \frac{2^{p}}{\varepsilon^{p}} V_{p}(\tilde{g})^{p} \right)^{\frac{1}{p}} & \\
& \leq \frac{4}{\varepsilon^{2}} 
\left( 
\left\| \tilde{g} \right\|_{\mathrm{L}^{\infty}([\varepsilon, R])}^{p} + \varepsilon^{p}  V_{p}(\tilde{g})^{p}
\right)^{\frac{1}{p}} &
\end{align*}
is true which gives 
$$
\left| \int_{\varepsilon}^{R} \tilde{g}(\mathrm{i}t) \left( \frac{1}{\mathrm{i}t} + \frac{\mathrm{i}t}{R^{2}} \right) 
\mathrm{e}^{\mathrm{i}tT}  \;dt \right|
\leq \frac{4}{T^{\frac{1}{p}} \varepsilon^{2}} \left( 
\left\| \tilde{g} \right\|_{\mathrm{L}^{\infty}([\varepsilon, R])}^{p} + \varepsilon^{p}  V_{p}(\tilde{g})^{p}
\right)^{\frac{1}{p}} + \frac{2}{T \varepsilon}  \left\| \tilde{g} \right\|_{\mathrm{L}^{\infty}([\varepsilon, R])}. 
$$
For a number $k \in (0, \frac{1}{2p})$ and large $T$ we choose $\varepsilon = T^{-k}$ to conclude to
$$
\left| \int_{T^{-k}}^{R} \tilde{g}(\mathrm{i}t) \left( \frac{1}{\mathrm{i}t} + \frac{\mathrm{i}t}{R^{2}} \right) 
\mathrm{e}^{\mathrm{i}tT}  \;dt \right| \leq 
\frac{4}{T^{\frac{1}{p}-2k}} \left( 
\left\| \tilde{g} \right\|_{\mathrm{L}^{\infty}([T^{-k}, R])}^{p} + \left( \frac{V_{p}(\tilde{g})}{T^{k}} \right)^{p}
\right)^{\frac{1}{p}} + \frac{2}{T^{1-k}}  \left\| \tilde{g} \right\|_{\mathrm{L}^{\infty}([T^{-k}, R])}.
$$ 
Please keep in mind that the bounded $p$-variation $ V_{p}(\tilde{g})$ of $\tilde{g}$ depends on $R$ and $T$ as well. 
The term 
$$
\left| \int_{-R}^{-T^{-k}} \tilde{g}(\mathrm{i}t) \left( \frac{1}{\mathrm{i}t} + \frac{\mathrm{i}t}{R^{2}} \right) 
\mathrm{e}^{\mathrm{i}tT}  \;dt \right|
$$
is treated similarily and therefore we have shown the claimed inequality.
\end{enumerate}

\subsection{Boundary behaviour of zeta functions of arithmetic semigroups}

Let $G$ be an arithmetic semigroup satisfying $N_G(x)=Ax+\mathcal{O}\left(\frac{x}{\log^\gamma x}\right)$ for some constant 
$A \in \mathbb{R}$. Let $\zeta_G(s)=\sum_{g\in G}\|g\|^{-s}$ be the associated zetafunction. To apply Theorem~\ref{thm:Laplace} we have to show that $\frac{\zeta'_G}{\zeta_G}$ or related functions behaves well on the line $1+it$. We do so by adapting Diamond's proof showing that under suitable conditions we have $\zeta(1+it)\neq 0$.\\

\begin{Lem}
\label{Lem:derivative}
Suppose that $G$ is an arithmetic semigroup satisfying 
$$
N_{G}(x)= Ax+\mathcal{O}\left(\frac{x}{\log^{\gamma} x}\right)
$$ 
with $\gamma>1$ and $A \in \mathbb{R}$. Then the following implications hold: \\
\begin{enumerate}
\item If $\gamma>1$ is fixed, then $|\zeta(\sigma+it)|\ll 1+\frac{1}{|\sigma+it-1|}+|t|$ holds for all $\sigma\geq 1$.
\ \\
\item If $1\leq \gamma<2$ is fixed, then
$|\zeta(\sigma+it)'| \ll 1 + t(\sigma-1)^{\gamma-2}$.
\ \\
\item If $\gamma=2$, then we have
$|\zeta(\sigma+it)'| \ll 1 + t|\log(\sigma-1)|$.
\ \\
\item If $1\leq\gamma< 3$ is fixed, then
$|\zeta(\sigma+it)''| \ll 1 + t(\sigma-1)^{\gamma-2}$. 
\ \\
\item If $\gamma=3$, then we have
$|\zeta(\sigma+it)''| \ll 1 + t|\log(\sigma-1)|$. \\
\end{enumerate}
\end{Lem}
\begin{proof}
By partial integration we have
\begin{multline*}
\zeta(\sigma+it) = (\sigma+it)\int_1^\infty \frac{N_{G}(x)}{x^{\sigma+it-1}}\; dx 
= (\sigma+it)\int_1^\infty \frac{Ax + \mathcal{O}\left(\frac{x}{\log^\gamma x}\right)}{x^{\sigma+it-1}}\; dx\\
 = \frac{A(\sigma+it)}{\sigma+it -1} + \mathcal{O}\left((1+|t|)\int_1^\infty\frac{dx}{x^\sigma \log^\gamma x}\right) \ll \frac{1}{\sigma+it -1} + 1+|t|.
\end{multline*}
Similarly we have
\begin{eqnarray*}
|\zeta(\sigma+it)'| & \ll &1+ \left|\int_e^\infty x^{-s}\frac{N_{G}(x)-Ax}{x} (|\sigma+it|\log x-1)\;dx\right|\\
& \ll & 1 + \left|\int_e^\infty \frac{x^{-\sigma}}{\log^\gamma x} (|\sigma+it|\log x-1)\;dx\right|\\
& \ll &  1 + t\left|\int_e^\infty \frac{x^{-\sigma}}{\log^{\gamma-1} x} \;dx\right|\\
 & \ll &  1 + t\left|\int_1^\infty e^{-(\sigma-1)u} u^{1-\gamma}\;du \right|\\
  & \ll & 1 + t(\sigma-1)^{\gamma-2}\left|\int_{\sigma-1}^\infty e^{-v} v^{1-\gamma}\;dv \right|\\
  & \ll & 1 + t(\sigma-1)^{\gamma-2} \left(1+ \frac{1}{2-\gamma}\right).
\end{eqnarray*}
If $\gamma<2$ is fixed, our claim follows. In the case of $\gamma=2$, we define 
$$
\gamma^{\ast}=2-\frac{1}{\log(\sigma-1)}.
$$
Then $N_{G}(x)= Ax+\mathcal{O}\left(\frac{x}{\log^{\gamma^{\ast}} x}\right)$ is implied by $\gamma > \gamma^{\ast}$. Hence
we use the argumentation above on $\gamma^{\ast}$ in place of $\gamma$.\\
Furthermore the proof for the second derivative differs only by one power of $\log x$.
\end{proof}

\begin{Lem}
Suppose that $G$ is an arithmetic semigroup satisfying 
$$
N_{G}(x)=Ax+\mathcal{O}\left(\frac{x}{\log^{\gamma} x}\right)
$$ 
with $1<\gamma\leq 2$ and $A \in \mathbb{R}$. Then the following implications hold: \\
\begin{enumerate}
\item If $\gamma<2$ is fixed, then $\zeta(1+it)\ll t^{1/\gamma}$ holds for $t>1$. 
\ \\
\item For $\gamma=2$ we have $\zeta(1+it) \ll \sqrt{t\log t}$.
\end{enumerate}
\end{Lem}
\begin{proof}
We have $|\zeta(s)| = \left|\sum_{g\in G}\frac{1}{\|g\|^s}\right| \leq \sum_{g\in G}\frac{1}{\|g\|^\sigma} \ll \frac{1}{\sigma-1}$. We now apply Lemma~\ref{Lem:derivative} to obtain
\begin{align*}
|\zeta(1+it)| & = \left|\zeta(\sigma+it) - \int_1^\sigma \zeta'(x+it)\;dx \right| \\
& \ll \frac{1}{\sigma-1} + t\int_1^\sigma (x-1)^{\gamma-2}\;dx \ll \frac{1}{\sigma-1} + t(\sigma-1)^{\gamma-1}.
\end{align*}
Choosing $\sigma-1=t^{-1/\gamma}$ our first claim follows. Similarly for $\gamma=2$ we have
\[
|\zeta(1+it)| \ll \frac{1}{\sigma-1} + t(\sigma-1)\log\frac{1}{\sigma-1},
\]
and the second claim follows by taking $\sigma-1=\frac{1}{\sqrt{t\log t}}$.
\end{proof}

\begin{Lem}
\label{Lem:Vallee Poussin}
Suppose that $G$ is an arithmetic semigroup satisfying 
$$
N_{G}(x) = Ax+\mathcal{O}\left(\frac{x}{\log^{\gamma} x}\right)
$$ 
with $\frac{3}{2}<\gamma<2$ and $A \in \mathbb{R}$. Then there exists a constant $c$ such that we 
have for $\sigma>1$ and an integer $n$ the estimate
\[
\zeta(\sigma+it) \gg (\sigma-1)^{\frac{1}{2}+\frac{2}{n-1}} (1+cn|t|^{1/\gamma})^{-n}
\]
\end{Lem}
\begin{proof}
Let $\varphi(x) = 1+\sum_{j=1}^n a_j\cos(jx)$ be a non-negative cosine polynomial. Then we have for $\sigma>1$
\begin{equation}
\label{eq:Vallee Poussin}
\log\zeta(\sigma)+a_1\log|\zeta(\sigma+it)| + \sum_{j=2}^n a_j \log|\zeta(\sigma+ijt)| \geq 0.
\end{equation}
By partical summation we have $\zeta(\sigma)=\frac{1}{\sigma-1}+\mathcal{O}(1)$. We now use Lemma~\ref{Lem:derivative} to obtain
\[
\log|\zeta(\sigma+it)| \geq -\frac{1}{a_1}\log(\sigma-1) + \sum_{j=2}^n |a_j|(\log (1+c(j|t|)^{1/\gamma}).
\]
We now choose $\varphi$ to be the Fejer-kernel
\[
\varphi(x) = 1 + 2\sum_{j=1}^n\left(1-\frac{j}{n}\right)\cos jx = \frac{1}{n}\left(\frac{\sin\frac{nx}{2}}{\sin\frac{x}{2}}\right)^2
\]
and obtain
\[
\log|\zeta(\sigma+it)| \geq -\frac{n}{2(n-1)}\log(\sigma-1) + n \log (1+cn|t|^{1/\gamma}).
\]
\end{proof}

\begin{Lem}
\label{Lem:inv bound}
Suppose that $G$ is an arithmetic semigroup satisfying 
$$
N_{G}(x)=Ax+\mathcal{O}\left(\frac{x}{\log^{\gamma} x}\right)
$$ 
with $\frac{3}{2}<\gamma\leq 2$ and $A \in \mathbb{R}$. Then we have
\[
|\zeta(1+it)|\gg t^{-\frac{8}{\left(\gamma-\frac{3}{2}\right)^2}}.
\]
\end{Lem}
\begin{proof}
For a fixed integer $n$ Lemma~\ref{Lem:Vallee Poussin} yields
\begin{eqnarray*}
|\zeta(1+it)| & = & \left|\zeta(\sigma+it) - \int_1^\sigma\zeta'(x+it)\;dx\right|\\
 & \geq & |\zeta(\sigma+it)| - ct\int_1^\sigma 1+(x-1)^{\gamma-2}\;dx|\\
 & \geq & c_1 (\sigma-1)^{\frac{1}{2}+\frac{2}{n-1}} (1+c_2n|t|^{1/\gamma})^{-n} - c_3t\max(\sigma-1, (\sigma-1)^{\gamma-1})\\
 & = & c_1 (\sigma-1)^{\frac{1}{2}+\frac{2}{n-1}} (1+c_2n|t|^{1/\gamma})^{-n} - c_3t(\sigma-1)^{\gamma-1}.
\end{eqnarray*}
As $n$ is fixed, and we may assume that $|t|\geq 2$, we obtain
\[
|\zeta(1+it)| \geq c_1' (\sigma-1)^{\frac{1}{2}+\frac{2}{n-1}} |t|^{-n/\gamma} - c_3t(\sigma-1)^{\gamma-1}.
\]
Write $\delta=\gamma-\frac{3}{2}$, and put $n=\lceil\frac{4}{\delta}\rceil+1$. Then $\frac{1}{2}+\frac{2}{n-1}\leq \frac{1+\delta}{2}=\gamma-1-\frac{\delta}{2}$ and therefore
\[
|\zeta(1+it)| \gg (\sigma-1)^{\frac{1+\delta}{2}}\left(t^{-\left(\frac{4}{\delta}+2\right)/\gamma} - t\frac{c_3}{c_1'}(\sigma-1)^{\frac{\delta}{2}}\right).
\]
If we pick $\sigma$ in such a way that $t\frac{c_3}{c_1'}(\sigma-1)^{\frac{\delta}{2}} = \frac{1}{2}t^{-\left(\frac{4}{\delta}+2\right)/\gamma}$, then we get
\[
|\zeta(1+it)| \gg (\sigma-1)^{\frac{1+\delta}{2}}t^{-\left(\frac{4}{\delta}+2\right)/\gamma} .
\]
The condition on $\sigma$ translates to $\sigma-1= \left(\frac{c_1'}{2c_3}t^{-(\frac{4}{\delta}-2)/\gamma-1}\right)^{\frac{2}{\delta}}$, that is, $\sigma-1 = c_4 t^{-\frac{8}{\gamma\delta^2}-\frac{4}{\gamma\delta}-\frac{2}{\delta}}$. Inserting this value we obtain
\[
|\zeta(1+it)|\gg t^{\left(-\frac{8}{\gamma\delta^2}-\frac{4}{\gamma\delta}-\frac{2}{\delta}\right)\frac{1+\delta}{2}} t^{-\frac{4}{\gamma\delta}-\frac{2}{\gamma}} \geq t^{\frac{3}{4}\left(-\frac{8}{\delta^2}-\frac{6}{\delta}\right)} t^{-\frac{4}{\delta}-2} = t^{-\frac{4}{\delta^2}-\frac{7}{\delta}-2} > t^{-\frac{8}{\delta}^2}.
\]
\end{proof}
\begin{Lem}
\label{Lem:line bound}
Suppose that $G$ is an arithmetic semigroup satisfying 
$$
N_{G}(x)=Ax+\mathcal{O}\left(\frac{x}{\log^{\gamma} x}\right)
$$ 
with $\frac{3}{2}<\gamma<2$ and $A \in \mathbb{R}$. Then we have
\[
|\zeta(1+it_1) - \zeta(1+i t_2)| \ll |t_1-t_2|^{\gamma-1} t_1^{2-\gamma}.
\]
If $2<\gamma<3$, then we have
\[
|\zeta(1+it_1)' - \zeta(1+i t_2)'| \ll |t_1-t_2|^{\gamma-2} t_1^{3-\gamma}.
\]
\end{Lem}
\begin{proof}
Let $p$ be the path going parallel to the real axis from $1+it_1$ to $\sigma+it_1$, parallel to the imaginary axis to $\sigma+it_2$, and parallel to the real axis to $1+it_2$. Then, using Lemma~\ref{Lem:derivative}, we have
\begin{eqnarray*}
|\zeta(1+it_1)-\zeta(1+it_2)| & = & \left|\int_p \zeta'(s)\;ds\right| \leq \int_p |\zeta'(s)|\;ds\\
 & \ll & 
 |t_1-t_2| + |\sigma-1| + \max(t_1, t_2)\int_1^\sigma (s-1)^{\gamma-2}\;ds + |t_1-t_2|(\sigma-1)^{\gamma-2}\\
 & \ll & |t_1-t_2| + |\sigma-1| + \max(t_1, t_2)(\sigma-1)^{\gamma-1} +  |t_1-t_2|(\sigma-1)^{\gamma-2}\\
 & \ll & t_1(\sigma-1)^{\gamma-1} +  |t_1-t_2|(\sigma-1)^{\gamma-2}\\
 & \ll &  t_1\left(\frac{|t_1-t_2|}{t_1}\right)^{\gamma-1}\\
  & = & |t_1-t_2|^{\gamma-1} t_1^{2-\gamma}.
\end{eqnarray*}
The same computation yields 
\begin{eqnarray*}
|\zeta(1+it_1)'-\zeta(1+it_2)'|  & \ll &  t_1(\sigma-1)^{\gamma-3} +  |t_1-t_2|(\sigma-1)^{\gamma-2} \\
 & \ll & t_1\left(\frac{|t_1-t_2|}{t_1}\right)^{\gamma-2}\\
 & = & |t_1-t_2|^{\gamma-2} t_1^{3-\gamma}.
\end{eqnarray*}
\end{proof}
\begin{Lem}
Suppose that $G$ is an arithmetic semigroup satisfying
$$
N_{G}(x) = Ax+\mathcal{O}\left(\frac{x}{\log^{\gamma} x}\right)
$$ 
with $2<\gamma<3$ and $A \in \mathbb{R}$. Then 
$$
\left\{ t\mapsto\frac{\zeta_G'(1+it)}{\zeta_G(1+it)} \right\}
$$ 
is $\gamma-2$-H\"older continuous, and the constant of the restrition of this function to the interval $[t, t-1]$ is $\mathcal{O}(t^{66})$.
\end{Lem}
\begin{proof}
Suppose that $|t_1-t_2|\leq 1$ and $t_1, t_2\geq 1$. Then we have
\begin{eqnarray*}
& & \left|\frac{\zeta'(1+it_1)}{\zeta(1+it_1)} - \frac{\zeta'(1+it_2)}{\zeta(1+it_2)}\right| 
= \frac{|\zeta(1+it_1)\zeta'(1+i t_2) - \zeta(1+i t_2)\zeta'(1+i t_1)|}{|\zeta(1+it_1)\zeta(1+i t_2)|} \\
\ \\
& & \leq  \frac{|\zeta(1+it_1)|\cdot|\zeta'(1+i t_2) - \zeta'(1+i t_1)| + 
|\zeta'(1+it_1)|\cdot|\zeta(1+it_1)-\zeta(1+it_2)|}{|\zeta(1+it_1)\zeta(1+i t_2)|} \\
\ \\
& & \ll  \frac{t_1^{\frac{1}{2}+\epsilon}|t_1-t_2|^{\gamma-2}t_1^{3-\gamma} + 
t_1^{\frac{1}{\gamma-1}} |t_1-t_2|^{1-\epsilon} t_1^\epsilon}{t_1^{-32-\epsilon}t_2^{-32-\epsilon}} \\
\ \\
& & \ll t_1^{66} |t_1-t_2|^{\gamma-2}, 
\end{eqnarray*}
which implies our claim. For $|t_1|, |t_2| \leq 1$ we can estimate the constant by $\mathcal{O}(1)$, and our claim follows. 
\end{proof}

We define 
$$
\Lambda(g)=
\left\{
\begin{array}{ll} 
\log \|p\|, & \text{$g$ is a power of a prime $p$} \\
 0, & \text{else}
\end{array}\right.
$$
and put $\psi(x)=\sum_{\|g\|\leq x} \Lambda(g)$. For $\Re s>1$ we have
\begin{equation}
\label{eq:Lambda series}
-\frac{\zeta'}{\zeta}(s) = \sum_g \frac{\Lambda(g)}{\|g\|^s} = \sum_g \Lambda(g)e^{-s\log\|g\|} = s\int_0^\infty \psi(e^x) e^{-sx}\;dx
\end{equation}
Bekehermes \cite[Lemma~3.4]{Tobias} proved the following.
\begin{Lem}
\label{Lem:zeta near 1}
Suppose that $\gamma>2$ and put $H(s) = -\frac{\zeta'}{\zeta}(s)-\frac{1}{s-1}$. Then the limit 
$\underset{\Re s>1}{\lim\limits_{s\rightarrow 1}}\frac{H(s)-H(1)}{s-1}$ exists. 
\end{Lem}
In particular we have that $\frac{H(s)-H(1)}{s-1}$ remains bounded in some neighbourhood of 1.
Finally Diamond\cite{Diamond} showed that $\gamma>1$ already implies $\psi(x)\ll x$. Hence, we can apply Theorem~\ref{thm:Laplace} to the function $f(t)=\frac{\psi(e^t)-e^t}{e^t}$ to find that we have for $2<\gamma<3$ the estimate
\[
\int_T^\infty f(t)\;dt \ll \left(\frac{R^{66}}{T}\right)^{\gamma-2} + \frac{R^{34}}{T} + \frac{1}{T^{\frac{\gamma-2}{3}}} + \frac{1}{R}.
\]
The first term always dominates the second one, to balance the first and the last one we have to pick $R=\frac{\gamma-2}{66(\gamma-2)+1}$ and obtain
\[
\int_T^\infty f(t)\;dt \ll T^{-\frac{\gamma-2}{66(\gamma-2)+1}} + T^{-\frac{\gamma-2}{3}} \ll T^{-\min\left(\frac{\gamma-2}{3}, \frac{1}{99}\right)}.
\]
Now suppose that $\psi(e^{t_0})>(1+\delta)e^{t_0}$. As $\psi$ is non-decreasing, we have $\psi(e^t)>(1+\frac{\delta}{2})e^t$ for $t\in[t_0, t_0+\log 1+\frac{\delta}{2}]>\frac{\delta}{2}$. Note that $\log(1+\frac{\delta}{2})>\frac{\delta}{3}$, provided that $\delta$ is sufficinetly small. Hence, on one hand we have
\begin{multline*}
\int_{t_0}^\infty \psi(e^t)-e^t\;dt - \int_{t_0+\frac{\delta}{3}}^\infty \psi(e^t)-e^t\;dt \leq\\
 \left|\int_{t_0}^\infty \psi(e^t)-e^t\;dt\right| + \left|\int_{t_0+\frac{\delta}{3}}^\infty \psi(e^t)-e^t\;dt\right| \ll  t_0^{-\min\left(\frac{\gamma-2}{3}, \frac{1}{99}\right)},
\end{multline*}
while on the other hand we have
\[
\int_{t_0}^\infty \psi(e^t)-e^t\;dt - \int_{t_0+\frac{\delta}{3}}^\infty \psi(e^t)-e^t\;dt = \int_{t_0}^{t_0+\frac{\delta}{3}} \psi(e^t)-e^t\;dt \geq \frac{\delta^2}{6}.
\]
Combining these estimates we find $\delta\ll t_0^{-\min(\frac{\gamma-2}{6}, \frac{1}{198})}$. We conclude that 
\[
\psi(e^t)\leq \left(1+t^{-\min(\frac{\gamma-2}{6}, \frac{1}{198})}\right) e^t.
\]
In the same we we obtain a lower bound: If $\psi(e^{t_0})$ is significantly smaller than $e^{t_0}$, then $\psi(e^t)$ is smaller than $e^t$ in a small interval to the left of $t$, and our claim follows in the same way as for the upper bound. Hence, we obtain that $\psi(x)=x+\mathcal{O}\left(\frac{x}{\log^\delta x}\right)$ with $\delta$ as in the theorem. We have
\begin{eqnarray*}
\psi(x) & \geq & \psi(x)-\psi\left(\frac{x}{\log^2 x}\right)\\
 &  \geq  & \left(\pi(x)-\pi\left(\frac{x}{\log^2 x}\right)\right)\log\left(\frac{x}{\log^2 x}\right)\\
 &  \geq & \left(\pi(x)-N\left(\frac{x}{\log^2 x}\right)\right)\log x - \pi(x)\log\log x\\
  & = & \pi(x)\log x + \mathcal{O}\left(\frac{x\log\log x}{\log x}\right).
\end{eqnarray*}
As $\delta<1$, we can use the upper bound for $\psi(x)$ to obtain $\pi(x)\leq\frac{x}{\log x} + \mathcal{O}\left(\frac{x}{\log^\delta x}\right)$. The corresponding lower bound follos in the same way, but is considerably easier, and the proof of Theorem~\ref{thm:PNT} is complete.

To prove Theorem~\ref{thm:Moebius} we first write $\frac{1}{\zeta(s)}$ similar to (\ref{eq:Lambda series}) as
\[
\frac{1}{\zeta(s)} = s\int_0^\infty M(e^t) e^{-ts}\;dt 
\]

\begin{Lem}
\label{Lem:Moebius continuity}
Suppose that $G$ is an arithmetic semigroup satisfying 
$$
N_{G}(x)=Ax+\mathcal{O}\left(\frac{x}{\log^{\gamma} x}\right)
$$ 
with $\frac{3}{2}<\gamma\leq 2$ and $A \in \mathbb{R}$. Then we have for $t_1, t_2>1$ with $|t_1-t_2|\leq 1$ the estimate
\[
\left|\frac{1}{\zeta(1+it_1)}-\frac{1}{\zeta(1+it_2)}\right| \ll  |t_1-t_2|^{\gamma-1} t_1^{\frac{17}{(\gamma-\frac{3}{2})^2}}
\]
\end{Lem}
\begin{proof}
Using Lemma~\ref{Lem:inv bound} and \ref{Lem:line bound} we have
\begin{multline*}
\left|\frac{1}{\zeta(1+it_1)}-\frac{1}{\zeta(1+it_2)}\right| = \frac{|\zeta(1+it_2)-\zeta(1+it_1)|}{|\zeta(1+it_1)\zeta(1+it_2)|} \ll \frac{|t_1-t_2|^{\gamma-1}t_1^{2-\gamma}}{t_1^{-\frac{8}{(\gamma-\frac{3}{2})^2}}t_2^{-\frac{8}{(\gamma-\frac{3}{2})^2}}}\\
\ll |t_1-t_2|^{\gamma-1} t_1^{2-\gamma+\frac{16}{(\gamma-\frac{3}{2})^2}} \ll 
|t_1-t_2|^{\gamma-1} t_1^{\frac{17}{(\gamma-\frac{3}{2})^2}}.
\end{multline*}
\end{proof}
Applying Theorem~\ref{thm:Laplace} as in the proof of Theorem~\ref{thm:PNT} we obtain
\[ 
\int_T^\infty \frac{M(e^t)}{e^t}\;dt \ll T^{-\frac{\gamma-1}{3}} + \frac{R^{\frac{17}{(\gamma-\frac{3}{2})^2}}}{T^{\gamma-1}} + \frac{R^{\frac{8}{(\gamma-\frac{3}{2})^2}}}{T}  + \frac{1}{R}.
\]
The second term always dominates the third done. To balance the second and third one we choose $R=x^{\frac{\gamma-1}{1+\frac{17}{(\gamma-\frac{3}{2})^2}}}$ and obtain
\[
\int_T^\infty  \frac{M(e^t)}{e^t}\;dt \ll  T^{-\frac{\gamma-1}{3}} + T^{-\frac{\gamma-1}{1+\frac{17}{(\gamma-\frac{3}{2})^2}}} \ll T^{-\frac{\gamma-1}{3}} + T^{-\frac{(\gamma-1)(\gamma-\frac{3}{2})^2}{18}} \ll T^{-\frac{(\gamma-\frac{3}{2})^2}{36}}.
\]
As $|M(x)-M(y)|\leq |N(x)-N(y)|$, we can pass from $\int M(e^t)$ to $M(x)$ as above, and Theorem~\ref{thm:Moebius} follows.

\appendix

\section{An effective Riemann-Lebesgue Lemma}\label{appendix}

For further reference let us recall the definition of the p-variation of a function on a totally ordered set. It is interpreted as
a measure of its regularity or smoothness.\\

\begin{Def}[p-variation]
Let $f \colon [0,1] \to \C$ be a function and $\mathcal{Z}[0,1]$ be the set of partitions of $[0,1]$, i.e.
$
\mathcal{Z}[0,1] = \left\{ (x_{k})_{k =0}^{n} \ : \ n \in \N, 0=x_{0} < x_{1} < \dots < x_{n-1} < x_{n} = 1 \right\}
$.\\
For $p \geq 1$ the p-variation of the function $f$ is defined as
$$
V_{p}(f) := \left( \sup_{(x_{k}) \in \mathcal{Z}[0,1]} \ \sum_{k=1}^n \left| f(x_k)-f(x_{k-1}) \right|^{p} \right)^{\frac{1}{p}} \in [0,\infty].\\
$$
The function $f$ admits a finite p-variation if $V_{p}$ is finite. The p-variation of a function decreases
with p. Furthermore functions with a finite $1$-variation are called functions of bounded variations. 
\end{Def}

Constant functions as well as step functions are simple examples of a finite p-variation for any $p \geq 1$. Also every function
$f \colon [0,1] \to \C$ admitting a finite p-variation is bounded since
$$
\left| f(x) \right| \leq \left| f(x) - f(0) \right| + \left| f(0) \right| \leq V_{p}(f) + \left| f(0) \right|.
$$
But apart from constant functions or step functions there are more suphisticated functions being of finite p-variation as 
we will see in the following.\\

\begin{Prop}
Let $\alpha \in (0,1]$ and $f \colon [0,1] \to \R$ be $\alpha$-H\"older continuous with a Hölder constant $L$. 
Then $f$ admits a finite p-variation for $p = 1/\alpha$ with $V_{p}(f) \leq L$.\\
\end{Prop}

\begin{proof}
Let $(x_{k}) \in \mathcal{Z}[0,1]$. Then
$$
\sum_{k=1}^n \left| f(x_{k})-f(x_{k-1}) \right|^{1/\alpha} \leq \sum_{k=1}^n \big( L \left| x_{k} - x_{k-1} \right|^\alpha \big)^{1/\alpha} 
= L^{1/\alpha}  \left( x_n-x_0 \right) = L^{1/\alpha} 
$$
is satisfied which implies the claim.\\
\end{proof}

Let us now prove an effective Riemann-Lebesgue Lemma for functions of a finite p-variation on $[0,1]$.\\

\begin{Theo}\label{Hoelder_variation}
Let $f \colon [0,1] \rightarrow \R$ be bounded and admit a finite p-variation $V_{p}(f)$. Then 
$$
\left| \int_0^1 f(t) e(xt)\;dt \right| \ \leq \  V_{p}(f)\left( \frac{1}{x} \right)^{\frac{1}{p}}+ \frac{\|f\|_{L^{\infty}}}{x}
$$
holds for any $x > 0$.\\
\end{Theo}

\begin{proof} \
\begin{enumerate}[i.)]
\item Let $x \in (0,1)$. Then
$$
\left|\int_0^1 f(t) e(xt)\;dt\right| \ \leq \ \int_{0}^{1} |f(t)| \;dt \ \leq \ \| f \|_{L^{\infty}} \ \leq \ \frac{ \| f \|_{L^{\infty}} }{x}
$$
is satisfied where the last inequality is due to $x<1$. Hence the claim is implied.\\
\item Let $x \geq 1$ and define $n=\lfloor x \rfloor$. Then $1 \leq n \leq x$ and
\begin{eqnarray*}
\left|\int_0^1 f(t) e(xt)\;dt\right|  & = & \left|\int_0^{n/x} f(t) e(xt)\;dt + \int_{n/x}^1 f(t) e(xt)\;dt \right| \\
 & \leq  &  \left| \int_0^{n/x} f(t) e(xt)\;dt \right| + \int_{n/x}^1 |f(t)|\;dt
 \end{eqnarray*}
are implied. Then we argue that
$$
\int_{n/x}^1 |f(t)|\;dt \leq \| f \|_{L^{\infty}} \left( 1-\frac{n}{x} \right)  = \|f\|_{L^{\infty}} \frac{x - \lfloor x\rfloor}{x} 
\leq \frac{\|f\|_{L^\infty}}{x}
$$ 
is true. For the first term we write

\begin{eqnarray*}
\int_0^{n/x} f(t) e(xt)\;dt & = &  \int_0^{n/x} \left( f \left(\frac{\lfloor xt \rfloor}{n}\right) + f(t) - f\left(\frac{\lfloor xt \rfloor}{n}\right)\right) e(xt)\;dt \\
  & = & \int_0^{n/x} f\left(\frac{\lfloor xt \rfloor}{n}\right) e(xt)\;dt + 
\int_0^{n/x} \left( f(t)-f\left(\frac{\lfloor xt \rfloor}{n}\right) \right) e(xt)\;dt \\
& = & \sum_{k=0}^{n-1} \int_{k/x}^{(k+1)/x}  f\left(\frac{k}{n}\right) e(xt)\;dt +
\int_0^{n/x} \left( f(t)-f\left(\frac{\lfloor xt \rfloor}{n}\right)\right) e(xt)\;dt \\
& = & \sum_{k=0}^{n-1}  f\left( \frac{k}{n} \right) \underbrace{\int_{k/x}^{(k+1)/x} e(xt)\;dt}_{=0}  
+ \int_0^{n/x} \left(f(t)-f\left(\frac{\lfloor xt \rfloor}{n}\right)\right) e(xt)\;dt. \\
 \end{eqnarray*}

Note that here $n$ is strictly greater than $0$ in this case such that terms like $\lfloor xt\rfloor/n$ and $k/n$
are well defined. Therefore $\int_0^{n/x} f(t) e(xt)\;dt$ is equal to 
$\int_0^{n/x} \left(f(t)-f\left(\frac{\lfloor xt \rfloor}{n}\right)\right) e(xt)\;dt$. Next, we apply Hölder's inequality 
for $p \geq 1$ to conclude to 

\begin{eqnarray*}
\left| \int_0^{n/x} f(t) e(xt)\;dt \right| & = & 
\left| \int_0^{n/x} \left( f(t)-f\left(\frac{\lfloor xt \rfloor}{n}\right) \right) e(xt)\;dt \right| \\
& \leq & \int_0^{n/x} \left| f(t)-f\left(\frac{\lfloor xt \rfloor}{n}\right) \right| \;dt \\
& \leq & \left( \int_0^{n/x} \left| f(t)-f\left(\frac{\lfloor xt \rfloor}{n}\right) \right|^{p} \;dt \right)^{\frac{1}{p}}  
\left(\int_0^{n/x} 1 \;dx\right)^{1-\frac{1}{p}}\\
& \leq & \left( \sum_{k=0}^{n-1} \int_{k/x}^{(k+1)/x} 
\left| f(t)-f\left(\frac{k}{n} \right) \right|^{p} \;dt \right)^{\frac{1}{p}}  \\
& \leq & \left( \frac{1}{x}  \sum_{k=0}^{n-1} \  
\sup_{t \in [k/x, (k+1)/x)} \left| f(t)-f\left(\frac{k}{n}\right) \right|^p \ \right)^{\frac{1}{p}} .\\
\end{eqnarray*}

For this last term we want to use the finite p-variation of $f$. Therefore let
$\varepsilon \in (0,1)$ be arbitrary but fixed and pick for each $k$ a real number 
$x_{k} \in \left[\frac{k}{t}, \frac{k+1}{t}\right)$ such that

$$
\left| f(x_{k}) - f\left(\frac{k}{n}\right) \right| \geq (1-\varepsilon) 
\sup_{t \in [\frac{k}{x}, \frac{k+1}{x})} \left|f(t)-f\left(\frac{k}{n}\right) \right|
$$
is true. Define $\delta_k = \left| f(x_{k})-f\left(\frac{k}{n}\right) \right|$ and conclude to 

$$
\sum_{k=0}^{n-1} \  
\sup_{t \in [\frac{k}{x}, \frac{k+1}{x})} \left|f(t)-f\left(\frac{k}{n}\right)\right|^{p} \leq 
\left( \frac{1}{1-\varepsilon} \right)^{p} \ \sum_{k=0}^{n-1} \delta_{k}^{p}. 
$$
We add the terms $\left| f\left(\frac{k+1}{n}\right) - f(x_{k}) \right|^{p}$ for each $k$ to this sum on the right 
to get a partition $\tilde{x}_{0} = 0, \ \tilde{x}_{1} = x_{0}, \ \tilde{x}_{2} = \frac{1}{n}, \ \tilde{x}_{3} = x_{1}, \ 
\tilde{x}_{4} = \frac{2}{n}, \ \dots$ on $[0,1]$ such that

$$
\left( \frac{1}{1-\varepsilon} \right)^{p} \ \sum_{k=0}^{n-1} \delta_{k}^{p} \leq 
\left( \frac{1}{1-\varepsilon} \right)^{p} \left( V_{p}(f) \right)^{p}
$$ 
is justified. Hence 
$\left| \int_0^{n/x} f(t) e(xt)\;dt \right| \leq \frac{V_{p}(f)}{1-\varepsilon} \left( \frac{1}{x} \right)^{1/p}$ and

\begin{equation*}
\left| \int_0^1 f(t) e(xt)\;dt \right| \leq \frac{V_{p}(f)}{1-\varepsilon} \left( \frac{1}{x} \right)^{\frac{1}{p}} + \frac{\|f\|_{L^\infty}}{x}
\end{equation*}
are implied. Letting $\varepsilon \to 0^{+}$ proves the claim.
\end{enumerate}
\end{proof}

Finally let us prove the following useful lemma for our argumentation.\\

\begin{Lem}
\label{Lem:products}
Let $g \colon [0,1] \rightarrow \C$ be of finite $p$-variation $V_{p}(g)$ and let $f \colon [0,1] \rightarrow \C$ be differentiable
with a bounded derivative on $[0,1]$. Then the product of these functions $fg$ is also of finite $p$-variation bounded such that 
$$
V_{p}(fg) \ \leq \ 2 \left( \| g \|_{L^\infty}^p \|f'\|_{L^\infty}^p + \|f\|_{L^\infty}^p V_{p}(g)^{p} \right)^{\frac{1}{p}}.
$$
\end{Lem}
\begin{proof}
Let $0 = x_0 < x_1 < \dots < x_n = 1$ be a partition of $[0,1]$. Then we have to estimate
\[
\sum_{k=1}^{n} \left| (fg)(x_{k}) - (fg)(x_{k-1}) \right|^{p}.
\]
We have
\begin{eqnarray*}
\left| (fg)(x_{k}) - (fg)(x_{k-1}) \right|^p & \leq & 
\left(  \left| g(x_{k}) \right| \cdot \left| f(x_{k}) - f(x_{k-1}) \right| + 
\left| f(x_{k-1}) \right| \cdot \left| g(x_{k})-g(x_{k-1}) \right|  \right)^{p} \\
& \leq & \left( \ 2 \cdot \max \left(  \left| g(x_{k}) \right| \cdot \left| f(x_{k}) - f(x_{k-1}) \right|,  
\left| f(x_{k-1}) \right| \cdot \left| g(x_{k})-g(x_{k-1}) \right|   \right) \ \right)^{p} \\
 & \leq & 
2^{p} \left( \
\|g\|_{L^\infty}^{p} \| f^{\prime} \|_{L^\infty}^{p} |x_{k} - x_{k-1}|^{p} + \|f\|_{L^\infty}^{p} \left| g(x_{k})-g(x_{k-1}) \right|^{p} \ \right)
\end{eqnarray*}
and hence
\begin{eqnarray*}
\sum_{k=1}^{n} |(fg)(x_{k}) - (fg)(x_{k - 1})|^p & \leq & 2^{p} 
\sum_{k=1}^n \| g \|_{L^\infty}^p \| f^{\prime} \|_{L^\infty}^p \left| x_k - x_{k-1} \right|^p + 
\|f\|_{L^\infty}^p \left| g(x_{k}) - g(x_{k-1}) \right|^p\\
& \leq & 2^{p} \left( \|g\|_{L^\infty}^p \| f^{\prime} \|_{L^\infty}^p + \|f\|_{L^\infty}^p V_{p}(g)^{p} \right) 
\end{eqnarray*}
\end{proof}


\begin{thebibliography}{9}
\bibitem{Tobias} T. Bekehermes, {\em Allgemeine Dirichletreihen und Primzahlverteilung in
arithmetischen Halbgruppen}, Clausthal, 2003.
\bibitem{Diamond} H. Diamond, The Prime Number Theorem for Beurling's Generalized Numbers, {\em J. Number Theory} {\bf 1} (1969), 200--207.
\bibitem{Diamond2} H. Diamond, Chebyshev estimates for Beurling generalized prime numbers, {\em Proceedings of the AMS} {\bf 39} (1973), 503--508.
\bibitem{Kahane} C. S. Kahane, 
Generalizations of the Riemann-Lebesgue and Cantor-Lebesgue lemmas, {\em Czechoslovak Mathematical Journal} {\bf 30} (1980), 108--117.
\bibitem{Riemenschneider} O. Riemenschneider, Simple analytic proofs of some versions of the abstract prime number theorem, in: Brasselet, Jean-Paul (ed.) et al., {\em Singularities}, Niigata-Toyama 2007, 249--283 
\end{thebibliography}
\end{document}